\documentclass[12pt]{amsart}
\usepackage{amsmath,amsxtra,amssymb,latexsym, amscd,amsthm}

\newtheorem{thm}{Theorem}[section]

\newtheorem{lem}[thm]{Lemma}
\newtheorem{prop}[thm]{Proposition}
\theoremstyle{definition}

\newtheorem{exm}[thm]{Example}
\newtheorem{rem}[thm]{Remark}

\numberwithin{equation}{section}

\DeclareMathOperator{\depth}{depth}
\DeclareMathOperator{\reg}{reg}

\def\Zset{\mathbb {Z}}

\def\Itil {\tilde{I}}

\def\mfr {\mathfrak m}

\begin{document}
\title[Hilbert coefficients] {Upper bounds on two Hilbert coefficients  }
\author{Le Xuan Dung$^1$, Juan Elias$^2$ and Le Tuan Hoa$^3$ }
\address{$^1$Faculty of  Natural Science,  Hong Duc University\\
565 Quang Trung, Dong Ve, Thanh Hoa, Vietnam}
\address{ $^2$Facultat de Matemàtiques i Informàtica, Universitat de Barcelona, Gran Via 585, 08007
Barcelona, Spain}
\address{  $^3$Institute of Mathematics, VAST, 18 Hoang Quoc Viet, 10307 Hanoi, Viet Nam}
\email{lexuandung@hdu.edu.vn, lxdung27@gmail.com (L. X. Dung), elias@ub.edu (J. Elias),  lthoa@math.ac.vn (L. T. Hoa)}
\date{}
\commby{}
%-----------------------------------------------------------
\begin{abstract}  New upper bounds on the first and the second Hilbert coefficients of a Cohen-Macaulay module over a local ring are given. Characterizations are provided for some upper bounds to be attained. The characterizations are given in terms of Hilbert series as well as in terms of the Castelnuovo-Mumford regularity of the associated graded module.
\end{abstract}
% -----------------------------------------------------------
\maketitle
% -----------------------------------------------------------

%%%%%%%%%%%%%%%%%%%%%%%%%%%%%%%%%%%%%%%%%%%%%%%%%%%%%%%%%%%%%%%%%%%%%%%%%%%%%%%%%%%%%%%
%%%%%%%%%%%%%%%%%%%%%%%%%%%%%%%%%%%%%%%%%%%%%%%%%%%%%%%%%%%%%%%%%%%%%%%%%%%%%%%%%%%%%%%
%%%%%%%%%%%%%%%%%%%%%%%%%%%%%%%%%%%%%%%%%%%%%%%%%%%%%%%%%%%%%%%%%%%%%%%%%%%%%%%%%%%%%%%
%%%%%%%%%%%%%%%%%%%%%%%%%%%%%%%%%%%%%%%%%%%%%%%%%%%%%%%%%%%%%%%%%%%%%%%%%%%%%%%%%%%%%%%
%%%%%%%%%%%%%%%%%%%%%%%%%%%%%%%%%%%%%%%%%%%%%%%%%%%%%%%%%%%%%%%%%%%%%%%%%%%%%%%%%%%%%%%
\section{Introduction}

Let $(A,\mfr)$ be a  Noetherian   local ring with an  infinite  residue field $ A/\mfr$ and $I$ an $\mfr$-primary ideal. Let $M$ be a finitely generated  $A$-module of dimension $d$. 
Then the Hilbert-Samuel function    $H^1_{I,M} (n) := \ell_A(M/I^{n+1}M)$ agrees with a polynomial, so-called  Hilbert-Samuel polynomial, $H\!P^1_{I,M} (n) $ for all $n\gg 0$. 
If we write 
$$H\!P^1_{I,M} (n) = e_0(I,M)\binom{n+d}{d} - e_1(I,M)\binom{n+d-1}{d-1} + \cdots + (-1)^d e_d(I,M),$$
then the integers $e_0(I,M), ..., , e_0(I,M)$ are called the {\it Hilbert coefficients of $M$ with respect to $I$}.
 
 There are intensive studies on finding bounds on several Hilbert coefficients. We refer the interested  readers to the book \cite{RV2} for main developments and rich references. 
 Our aim is to give  new bounds on the first two Hilbert coefficients in terms of $e_0(I,M)$ for Cohen-Macaulay modules. 
Recall that $e_0(I,M), e_1(I,M), e_2(I,M)\ge 0$ and that $e_3(I,M)$ could be negative, see \cite{Na} and \cite{RV2}.
 
In this paper we establish two upper bounds on $e_1(I,M)$. 
The first one (Proposition \ref{H1}) holds for Cohen-Macaulay modules and is a slight improvement of the bound given by Rossi and Valla for filtrations in \cite[Proposition 2.8 and Proposition 2.10]{RV2}. 
 Moreover, in the case of dimension one, we can give conditions for this bound to be attained (see Proposition \ref{H3} and Proposition \ref{T2}). 
 In difference to the approach in \cite[Section 2.2]{RV2}, we use here local cohomology modules of the associated graded modules $G_I(M)$. 
 As a consequence,  we can show that  if $IM \subseteq\mfr^bM$ for some $b\ge 2$, then 
 $e_1(I,M) \le \binom{e_0(I,M) -b}{2}$, see Proposition \ref{H2}.  
 In the ring case, this result was given by  Elias \cite[Proposition 2.5 and Remark 2.6]{E1} under an additional condition, which was then removed by Hanumanthu and Huneke in \cite[Corollary 3.7]{HH}. 
 This part can be also seen as a preparation for our study on the $e_2(I,M)$ in Section \ref{e2Mo}.
 
 The second upper bound on $e_1(I,M)$ (Theorem \ref{EB3}) only holds for $M=A$ and is based on the bound given by Elias \cite{E2} in the dimension one case. 
 When $b\ge 2$, this new bound provides a much better bound than the above mentioned bound 
 $\binom{e_0(I,A) -b}{2}$, see Remark \ref{EB4}. 

For $e_2(I,M)$ we can give a  new upper bound in terms of $e_0(I,M)$ and $b$, see Theorem \ref{B1}. As we know that the first upper bound on $e_2(I,M)$ was given by Rhodes, see \cite[Proposition 6.1(iv)]{Rh} (also see \cite[Corollary 4.2]{KM} for modules over a semi-local ring). 
In \cite[Theorem 2.3]{ERV1}  a much better bound is given. However, all bounds on $e_2(I,M)$ in \cite{Rh, KM, ERV1} involve $e_1(I,M)$, while our bound  only depend on $e_0(I,M)$ (and the largest number $b$ such that $IM \subseteq \mfr^bM$). 
Moreover, we can characterize when this bound is attained (Theorem \ref{B2}). Like  Proposition \ref{H3},  the conditions in Theorem \ref{B2} are given in terms of Hilbert series as well as in terms of the Castelnuovo-Mumford regularity of the associated graded module $G_I(M)$.  
In the case $M=A$ and $b\ge 2$, using Theorem \ref{EB3} one can get a better bound than the one in Theorem \ref{B1},  see  Theorem \ref{AB8}.

\medskip
We now give a brief content of the paper. In Section \ref{P} we recall some basic notions and  give some estimations on the Hilbert function of $G_I(M)$ and of $\overline{G_I(M)}$. 
In Section \ref{e1Mo} we give two bounds on $e_1(I,M)$ (see Propositions \ref{H1} and \ref{H2}) and characterize when the first bound in Propositions \ref{H1} is attained, provided $\dim M= 1$. In Section \ref{e1Id} we restrict to the case $M=A$. 
Here we give further structures of $I$ and $A$ such that the first bound in Propositions \ref{H1} is attained, see Proposition \ref{T2}. 
Then we prove an essentially new bound on $e_1(I)$ (Theorem \ref{EB3}). 
Main known upper bounds on $e_1(I)$ of an $\mfr$-primary ideal $I$ of an one-dimensional Cohen-Macaulay ring $(A,\mfr)$ such that $I\varsubsetneq \mfr^2$  are summarized  in Remark \ref{EB5}. 
In the last Section \ref{e2Mo}, we prove the new bounds on $e_2(I,M)$ (Theorems \ref{B1} and \ref{AB8}), and give equivalent conditions for the  bound in Theorem \ref{B1} to be  attained  (Theorem \ref{B2}).

%%%%%%%%%%%%%%%%%%%%%%%%%%%%%%%%%%%%%%%%%%%%%%%%%%%%%%%%%%%%%%%%%%%%%%%%%%%%%%%%%%%%%%%%%%%%%%%%%%%%
%%%%%%%%%%%%%%%%%%%%%%%%%%%%%%%%%%%%%%%%%%%%%%%%%%%%%%%%%%%%%%%%%%%%%%%%%%%%%%%%%%%%%%%%%%%%%%%%%%%%
%%%%%%%%%%%%%%%%%%%%%%%%%%%%%%%%%%%%%%%%%%%%%%%%%%%%%%%%%%%%%%%%%%%%%%%%%%%%%%%%%%%%%%%%%%%%%%%%%%%%
%%%%%%%%%%%%%%%%%%%%%%%%%%%%%%%%%%%%%%%%%%%%%%%%%%%%%%%%%%%%%%%%%%%%%%%%%%%%%%%%%%%%%%%%%%%%%%%%%%%%
%%%%%%%%%%%%%%%%%%%%%%%%%%%%%%%%%%%%%%%%%%%%%%%%%%%%%%%%%%%%%%%%%%%%%%%%%%%%%%%%%%%%%%%%%%%%%%%%%%%%
\medskip
\section{Preliminaries} \label{P}

Let $R= \oplus_{n\ge 0}R_n$ be a Noetherian standard graded ring over a local Artinian ring $(R_0,\mfr_0)$.  
Let $E$ be a finitely generated graded module of dimension $d$. 
 The function $H_E(n) := \ell_{R_0}(E_n)$ is called {Hilbert function} of $E$. 
 For all $n\gg 0$, it agrees with the so-called {\it Hilbert polynomial} denoted by $H\!P_E(t)$, that is polynomial of degree $d-1$. 
 The number
 $$pn(E) :=  \min\{n |\  H_E(t) =  H\!P_E(t) \ \text{for all}\ t\ge n \},$$
 is called the {\it postulation number} of $H_E$.
 
If we denote by $R_+ := \oplus_{n> 0}R_n$ the irrelevant ideal of $R$ then we set
$$ a_i (E) := \sup\{ n|\ H^i_{R_+}(E)_n \neq 0\},$$
$0\le i \le d$.
The {\it Castelnuovo-Mumford regularity} of $E$ is the number:
$$\reg(E) := \max\{a_i(E) + i|\ 0\le i \le d\}.$$

\medskip
Let $(A,\mfr)$ be a  Noetherian   local ring with an  infinite  residue field $ A/\mfr$ and $I$ an $\mfr$-primary ideal. 
Let $M$ be a finitely generated  $A$-module of dimension $d$. 
The associated graded module of $I$ with respect to $M$ is the standard $A/I$-algebra
$$G_I(M) := \oplus_{n\ge 0}I^nM/I^{n+1}M,$$

There is another way to define the Hilbert coefficients $e_i(I,M)$ already defined in the introduction. 
We recall here this approach from \cite[Section 1.3]{RV2}. 
The function
$$H_{I,M} (n):= H_{G_I(M)}(n)=\ell_A(I^nM/I^{n+1}M)$$
is called the {\it Hilbert function} of $M$. 
The Hilbert polynomial of $M$ is $H\!P_{I,M}=H\!P_{G_I(M)}$.

By the Hilbert-Serre theorem, the  Hilbert series 
\begin{equation}\label{ser} P_{I,M}(z) : = \sum_{n\ge 0} H_{I,M} (n) z^n,
\end{equation} 
is a rational function of $z$, that means one can find a polynomial $Q_{I,M}(z)  \in \Zset[z]$ such that $Q_{I,M}(1) \neq 0$ and 
$$P_{I,M}(z) = \frac{Q_{I,M}(z)}{(1-z)^d}.$$
If we set for every  $i\ge 0$ 
\begin{equation}\label{EP1}e_i(I,M) = \frac{Q_{I,M}^{(i)}(1)}{i!},
\end{equation}
where $Q_{I,M}^{(i)}$ denotes the $i$-th derivation of $Q_{I,M}$, then for all $0\le i \le d$, this value of $e_i(I,M)$ agrees with the one defined in the introduction. 
Moreover, unlike the definition in the introduction, using  (\ref{EP1}) we can talk about the Hilbert coefficients $e_i(I,M)$ with $i>d$. 
This simple observation is useful in the study of the second Hilbert coefficient, where we can reduce the case of dimension two to dimension one.

Together with Hilbert series, the power series
$$P^1_{I,M}(z) : = \sum_{n\ge 0} \ell(M/I^{n+1}M)z^n = \frac{P_{I,M}(z)}{(1-z)^{d+1}}$$
is also often used; this is called {\it Hilbert-Samuel series}.

In the sequel we use the following notations
$$pn(I,M) := pn(G_I(M)) = \min\{n|\ H\!P_{I,M}(t) = H_{I,M}(t) \ \text{for all}\ t\ge n\},$$
\noindent
and if $M=A$ then we write $H_I:=H_{I,A}$, 
$H\!P_I:=H\!P_{I,A}$, $pn(I):= pn(I,A)$,  $e_i(I) := e_i(I,A)$  and so on.

\medskip
Recall that an element $x\in I$ is called {\it $M$-superficial (of order one) for} $I$,  if there exists a non-negative integer $c$ such that
$$(I^{n+1}M:x) \cap I^cM = I^nM,$$
for all $n\ge c$.  
When $M= A$ we simply say that $x$ is a superficial element for $I$. 
This is equivalent to the condition that the initial form $$x^* \in G(I):= \oplus_{n\ge 0}I^n/I^{n+1}$$
has degree one and it is a filter-regular element on the associated graded module
$G_I(M)$ which means
$$[ 0:_{G_I(M)} x^*]_m = 0 \ \text{for all}\ m\gg 0.$$
(See, e.g., the equivalence of (1) and (5) in \cite[Theorem 1.2]{RV2}.)  
A sequence $x_1,...,x_s\in I$ is called {\it $M$-superficial sequence for $I$} if $x_i$ is an  $M/(x_1,...,x_{i-1})M$-superficial element, for $I$ for $i=1,...,s$. 

The following result is now standard and useful for proceeding by induction.

\medskip
\begin{lem}\label{I2}{\rm (See, e.g.,  \cite[Proposition 1.2]{RV1})}  Let $x \in I $ be an $M$-superficial element for $I$. Then
\begin{itemize}
 \item[(i)]  $\dim(M/xM) = d - 1,$
 \item[(ii)]  $e_j (I,M/xM) = e_j(I,M)$ for every $j = 0,...,d-2,$
 \item[(iii)]  $e_{d-1}(I,M/xM) = e_{d-1}(I,M) + (-1)^{d-1}\ell(0:x),$
 \item[(iv)]  There exists an integer  $n_0$ such that for every  $n \geq n_0-1$ we have
$$ e_d(I,M/xM) = e_d(I,M) +  (-1)^d\left[\sum_{i=0}^n\ell(I^{n+1}M:x/I^nM) - (n+1)\ell(0:x)\right],$$
 \item[(v)] $x^*$ is a regular element on $G_I(M)$ if only if  $P_{I,M}(z) = P^1_{I,M/xM}(z) = \frac{P_{I,M/xM}(z)}{1-z}$ if only if $x$ is $M$-regular and $e_d(I,M) = e_d(I,M/xM)$.
\end{itemize}
\end{lem}

We would like to comment that in the above statements (iv) and (v), $e_d(I,M/xM)$ is the one defined by (\ref{EP1}).

The Ratliff-Rush closure of an ideal introduced in \cite{RR} plays an important role in the study of Hilbert functions, see e.g. \cite{RV1, RV2}. 
It is defined by
\begin{equation}\label{EP2}
\widetilde{I^nM} = \bigcup_{k\ge 1} I^{n+k}M : I^k = I^{n+l}: I^l \ \text{for some}\ l\gg 0.
\end{equation}
Using this notion, we can compute the zero-th local cohomology module of $G_I(M)$ with respect to $G_+:= \oplus_{n\ge 1}I^n/I^{n+1}$ as follows (see, e.g., \cite[p. 26]{RV1}):
\begin{equation}\label{EP3}
[H^0_{G_+}(G_I)M))]_n = \frac{\widetilde{I^{n+1}M} \cap I^n M}{I^{n+1}M}. 
\end{equation}
We set $\overline{G_I(M)} = G_I(M)/ H^0_{G_+}(G_I(M))$.

\medskip
\begin{lem} \label{I4} Let $M$ be an one-dimensional $A$-module. Let $b$ be a positive integer such that $IM \subseteq \mfr^b M$. 
Then
\begin{itemize}
\item[(i)] (\cite[Lemma 2.5]{D}) $\ell(\overline{G_I(M)}_0) \ge b$.
\item[(ii)] If $e_0(I,M) \neq e_0(\mfr^b, M)$, then $\ell(\overline{G_I(M)}_0) \ge b+1$.
\end{itemize} 
\end{lem}
\begin{proof} (i) is \cite[Lemma 2.5]{D}. 
It is based on the the fact $\ell(\overline{G_I(M)}_0) = \ell(E)$, where $E=M/\widetilde{IM}$, and the strict inclusions:
\begin{equation}\label{EP4}
E \supsetneq \mfr E \supsetneq \cdots \supsetneq \mfr^b E.
\end{equation}

\noindent
(ii) Assume that $\mfr^bE = 0$. 
Then $\mfr^b M\subseteq \widetilde{IM}$. By (\ref{EP2}), it implies that
$$I^{l+1}M \subseteq I^l( \mfr^b M ) \subseteq I^{l+1}M,$$
for some $l\gg 0$. Hence $I^{l+1}M = \mfr^bI^lM$. 
Let $c$ be an integer such that  $\mfr^{bc} \subseteq I^l$, then for all $n>0$, it yields 
$$I^{l+n} M = (\mfr^b)^n I^lM \supseteq (\mfr^b)^{n +c}M.$$
Hence 
$$\begin{array}{ll} (n+l) e_0(I,M)  - e_1(I,M) & = \ell(M/I^{l+n} M) \\ 
& \le \ell(M/(\mfr^b)^{n +c}M) = (n+c)e_0(\mfr^b,M) - e_1(\mfr^b,M),
\end{array}$$
for all $n\gg 0$. 
This implies $e_0(I,M) \le e_0(\mfr^b,M)$, whence $e_0(I,M) = e_0(\mfr^b,M)$, a contradiction to the assumption. So, we must have $\mfr^bE \neq 0$. 
From (\ref{EP4}) we then get $\ell(\overline{G_I(M)}_0) = \ell(E) \ge b+1$, as required.
\end{proof}

\medskip
 \begin{lem} \label{I5} Let $R$ be a Noetherian standard graded ring over a local Artinian ring and $E$  an one-dimensional  Cohen-Macaulay graded $R$-module. 
 Let $\Delta: = \Delta(E)$ be the maximal generating degree of $E$. 
 Then
 \begin{itemize}
\item[(i)] $H_E(\Delta) < H_E(\Delta + 1) < \cdots < H_E(pn(E))$,
\item[(ii)] $H_E(n) \geq (n- \Delta ) +  H_E(\Delta ) $ for all $\Delta +1 \leq n \leq pn(E).$
\end{itemize} 
\end{lem}
\begin{proof} 
Let  $z\in R_1$  be an $E$-regular element.  
Since $E$ is a Cohen-Macaulay module,  $pn(E) = pn(E/zE) - 1$.  
Note  that $\Delta(E/zE) \le \Delta = \Delta(E)$ and  since $\dim E/zE =0$, $H_{E/zE}(t) \ge 1$ for all  $\Delta(E/zE) \le t \le pn(E/zE) -1 = pn(E) $. 
Hence the statements follow from the following equality
$$H_E(n) = H_E(\Delta ) + \sum_{\Delta +1\leq  i \leq n}H_{E/zE}(i).$$
\end{proof}
 
 The following result is a slight improvement of  \cite[Proposition 2.7]{RV2} and the remark after it.

\begin{lem} \label{I6} Let  $M$ be an one-dimensional  Cohen-Macaulay $A$-module such that $IM \subseteq \mfr^b M$.  Then
\begin{itemize}
\item[(i)] $ a_1(G_I(M)) = pn(I,M) - 1 > a_0(G_I(M))$ and $\reg (G_I(M)) = pn(I,M)$,
\item[(ii)]   $H_{I,M}(n) \ge n+b + \ell(H^0_{G_+}(G_I(M))_n)$ for all $  0\le n \le pn(I,M)$,
\item[(iii)] {\rm (\cite[Proposition 2.7(2)]{RV2})}  $pn(I,M) \le e_0(I,M) - b$. 
If the equality holds, then $\ell(\overline{G_I(M)}_n) = n+b$ for all $  0\le n \le pn(I,M)$.
\end{itemize}

\noindent
If,  in addition, $e_0(I,M) \neq e_0(\mfr^b,M)$, then we further have:
\begin{itemize}
\item[(iv)] $H_{I,M}(n) \ge n+b +1+ \ell(H^0_{G_+}(G_I(M))_n)$ for all $  0\le n \le pn(I,M)$,
\item[(v)] {\rm (\cite[Remark (b) after Proposition 2.7]{RV2}) } 
$pn(I,M) \le e_0(I,M) - b - 1$. 
If the equality holds, then $\ell(\overline{G_I(M)}_n) = n+b+1$ for all $  0\le n \le pn(I,M)$.
\end{itemize}
\end{lem}
\begin{proof} (i) By \cite[Theorem 5.2]{H}, $a_0(G_I(M
)) < a_1(G_I(M))$. From the Grothendieck-Serre  formula
\begin{multline*}
    H_{I,M}(n)- H\!P_{I,M}(n) = H_{G_{I,M}}(n) - H\!P_{G_{I,M}}(n)=\\
    =\ell(H^0_{G_+}(G_I(M))_n)-\ell(H^1_{G_+}(G_I(M))_n),
\end{multline*}
it follows that $pn(G_I(M)) = a_1(G_I(M)) + 1$.

\noindent
(ii)  From the short exact sequence
$$ 0 \longrightarrow H^0_{G_+}(G_I(M)) \longrightarrow G_I(M) \longrightarrow \overline{G_I(M)}:=G_I(M)/ H^0_{G_+}(G_I(M))  \longrightarrow 0,$$
we get  
$$H_{I,M}(n) = \ell(\overline{G_I(M)}_n) +  \ell(H^0_{G_+}(G_I(M))_n).$$
From (i) it follows that   
$$pn(I,M) = a_1(G_I(M)) +1 = pn (\overline{G_I(M)}).$$
Note that $\Delta (G_I(M)) =0$, and by Lemma \ref{I4}(i), $\ell(\overline{G_I(M)}_0) \ge b$. 
Since $\overline{G_I(M)}$ is a Cohen-Macaulay module, by Lemma \ref{I5}, we have
\begin{eqnarray} H_{I,M}(n) & =& \ell(H^0_{G_+}(G_I(M))_n) + H_{\overline{G_I(M)}}(n)  \nonumber \\
& \ge & \ell(H^0_{G_+}(G_I(M))_n) + H_{\overline{G_I(M)}}(0) + n \label{EI61}\\
& \ge & \ell(H^0_{G_+}(G_I(M))_n)  + b + n, \label{EI62}
\end{eqnarray}
for all $0\leq n \leq pn(\overline{G_I(M)})  = pn(G_I(M)) = pn(I,M)$.

\noindent
(iii) Since $H_{I,M}(n) \leq e_0(I,M)$ (see, e.g.,  the remark after Lemma 2.1 in \cite{RV2}), from (ii) we immediately get 
$$p \le e_0(I,M) - b - \ell(H^0_{G_+}(G_I(M))_p) \le e_0(I,M) - b,$$
 where  $p := pn(I,M)$. 
 If the equality holds, then from (\ref{EI61}) and (\ref{EI62}) we must have $H_{\overline{G_I(M)}}(0) = b$ and $H_{\overline{G_I(M)}}(p) = p+b$. 
 From Lemma \ref{I5} we then get 
 $$\ell(\overline{G_I(M)}_n) =H_{\overline{G_I(M)}}(n)= n+b$$
 for all $ 0\le n \le p$.
 
 \noindent
 (iv) and (v) If,  in addition, $e_0(I,M) \neq e_0(\mfr^b,M)$, then  using Lemma \ref{I5}(ii), instead of (\ref{EI62}), we get a little bit stronger inequality: 
$$ H_{I,M}(n) \ge \ell(H^0_{G_+}(G_I(M))_n)  + b + 1 +n,$$
which then implies (iv) and (v).
\end{proof}

%%%%%%%%%%%%%%%%%%%%%%%%%%%%%%%%%%%%%%%%%%%%%%%%%%%%%%%%%%%%%%%%%%%%%%%%%%%%%%%%%%%%%%%%%%%%%%%%%%%%%
%%%%%%%%%%%%%%%%%%%%%%%%%%%%%%%%%%%%%%%%%%%%%%%%%%%%%%%%%%%%%%%%%%%%%%%%%%%%%%%%%%%%%%%%%%%%%%%%%%%%%
%%%%%%%%%%%%%%%%%%%%%%%%%%%%%%%%%%%%%%%%%%%%%%%%%%%%%%%%%%%%%%%%%%%%%%%%%%%%%%%%%%%%%%%%%%%%%%%%%%%%%
%%%%%%%%%%%%%%%%%%%%%%%%%%%%%%%%%%%%%%%%%%%%%%%%%%%%%%%%%%%%%%%%%%%%%%%%%%%%%%%%%%%%%%%%%%%%%%%%%%%%%
%%%%%%%%%%%%%%%%%%%%%%%%%%%%%%%%%%%%%%%%%%%%%%%%%%%%%%%%%%%%%%%%%%%%%%%%%%%%%%%%%%%%%%%%%%%%%%%%%%%%%
\medskip
\section{The first Hilbert coefficient of a module} \label{e1Mo}

In this section we always assume that $M$ is a Cohen-Macaulay module over a local ring $(A,\mfr)$ and $I$ is an $\mfr$-primary ideal. 
We start with a slight improvement of \cite[Proposition 2.8 and Proposition 2.10]{RV2}. 
Its proof is also a modification of the one of \cite[Proposition 2.8]{RV2}. 
Note that \cite[Proposition 2.8 and Proposition 2.10]{RV2} is formulated for an arbitrary filtered module (not necessarily Cohen-Macaulay). 

\begin{prop} \label{H1} 
Let  $M$ be a  Cohen-Macaulay module and of dimension $ d \geq 1$. 
Let  $b$ be a positive integer such that $IM \subseteq \mfr^b M.$ 
Then
\begin{equation}\label{EH11} e_1(I,M) \leq  \binom{e_0(I,M) - b +1}{2}+b-\ell(M/IM) .
\end{equation}
If  $d=1$ and the equality in {\rm{(\ref{EH11})}} holds, then we have
\begin{itemize}
\item[(i)] $a_0(G_I(M) )\le 0$,
\item[(ii)]  Either $\reg (G_I(M)) = pn(I,M) = e_0(I,M) -b$ or $e_0(I,M) \in \{b, b  + 1\}$,
\item[(iii)] $H_{I,M}(n) = b+n$ for all $1 \le n \le pn(I,M)-1$.
\end{itemize} 
If $d=1$ and $e_0(I,M) > e_0(\mfr^b, M)$, then
\begin{equation}\label{EH12} e_1(I,M) \leq  \binom{e_0(I,M) - b}{2}+b +1 -\ell(M/IM) .
\end{equation}
If the equality in {\rm{(\ref{EH12})}} holds, then we have
\begin{itemize}
\item[(i')] $a_0(G_I(M) )\le 0$,
\item[(ii')]  Either $\reg (G_I(M)) = pn(I,M) = e_0(I,M) -b -1 $ or $e_0(I,M) \in \{b+1, b  + 2\}$,
\item[(iii')] $H_{I,M}(n) = n+b+1$ for all $1 \le n \le pn(I,M)-1$.
\end{itemize} 
\end{prop}
\begin{proof} For simplicity, set $e_0:=e_0(I,M),\ e_1 := e_1(I,M)$ and $p := pn(I,M) $. By  standard technique (using Lemma \ref{I2}) we may assume that  $d = 1$.    We have 
$$e_1 = \sum_{i=0}^{p-1} (e_0- H_{I,M}(i)) = e_0- H_{I,M}(0)+\sum_{i=1}^{p-1} (e_0-H_{I,M}(i)).$$
Using Lemma \ref{I6}(ii) and (iii), we get
\begin{eqnarray}\nonumber 
e_1 &\le& e_0- H_{I,M}(0)+\sum_{i=1}^{p-1} (e_0-i-b - \ell(H^0_{G_+}(G_I(M))_i) \\
& \le & \sum_{i=1}^{e_0-b-1}(e_0-i-b - \ell(H^0_{G_+}(G_I(M))_i )  + e_0 - \ell(M/IM) \label{EH13} \\
& \le & \sum_{i=1}^{e_0-b-1}(e_0-i-b )  + e_0 - \ell(M/IM) \label{EH14}\\
&= & \binom{e_0-b+1}{2} + b - \ell(M/IM). \nonumber
\end{eqnarray}
If $e_1 = \binom{e_0-b+1}{2} + b - \ell(M/IM)$, then from (\ref{EH13}) and (\ref{EH14}) we must have:

\medskip
\begin{itemize}
\item[(a)] $H^0_{G_+}(G_I(M))_i  =0$ for all $1\le i \le e_0-b-1$,
\item[(b)] $H_{I,M}(i) = b+i$ for all $1\le i \le e_0-b-1$,
\item[(c)] $p = e_0-b$ if $e_0-b\ge 2$.
\end{itemize} 

\medskip
\noindent
Since $p\le e_0-b$ by Lemma \ref{I6}(iii), (b) implies (iii). By Lemma \ref{I6}(i), $p-1 > a_0(G_I(M))$. Hence (a) implies (i). Since $a_1(G_I(M)) > a_0(G_I(M))$ (by Lemma \ref{I6}(i)), $\reg(G_I(M)) = a_1(G_I(M)) +1$.  Using again Lemma \ref{I6}(i), we get $\reg(G_I(M)) = p$. Then (c) implies (ii).

Finally, if $e_0(I,M) > e_0(\mfr^b, M)$, then by Lemma \ref{I6}(v), $p\le e_0-b-1$. 
Hence as above, we get
\begin{eqnarray}
e_1 & \le & \sum_{i=1}^{e_0-b-2}(e_0-i-b-1 - \ell(H^0_{G_+}(G_I(M))_i )  + e_0 - \ell(M/IM) \label{EH15} \\
& \le & \sum_{i=1}^{e_0-b-2}(e_0-i-b-1 )  + e_0 - \ell(M/IM) \label{EH16}\\
&= & \binom{e_0-b}{2} + b + 1- \ell(M/IM) . \nonumber
\end{eqnarray}
The proof of (i'), (ii') and (iii') is similar to that of (i), (ii) and (iii), where (\ref{EH15}) and (\ref{EH16}) are used.
\end{proof}

\medskip
Assume that $A$ is a Cohen-Macaulay ring. Elias \cite[Proposition 2.5]{E1} showed that if $I \subseteq \mfr^b$ for some $b\ge 2$, then under an additional condition we have:
$$e_1(I) \le \binom{e_0(I)-b}{2}.$$
Using integral closures of an ideal, K. Hanumanthu and C. Huneke were able to remove that additional condition (see \cite[Corollary 3.7]{HH}). We can now extend this result to the case of modules.

\begin{prop} \label{H2} 
Assume that $M$ is a Cohen-Macaulay $A$-module of positive dimension and $I$ is an $\mfr$-primary ideal of $(A,\mfr)$ such that $IM \subseteq \mfr^b M$ for some $b\ge 2$. 
Then
$$e_1(I,M) \le \binom{e_0(I,M) - b}{2}.$$
\end{prop}
\begin{proof} Using standard technique  we may assume that  $d = 1$. 
If $e_0(I,M) > e_0(\mfr^b,M)$, then the statement follows from Proposition \ref{H1} (\ref{EH12}), since $$\ell(M/IM) > \ell(M/\mfr^b M) \ge b.$$

Assume now that $e_0(I,M) =  e_0(\mfr^b,M)$.  
For $n\gg 0$, we have

$$\begin{array}{ll} e_0(I,M) (n+1) - e_1(I,M) = 
\ell (M/I^{n+1}M) & \ge  \ell (M/(\mfr^b)^{n+1}M) \\ \\
& = e_0(\mfr^b,M)(n+1) - e_1(\mfr^b,M).\end{array}$$
Hence $e_1(I,M) \le e_1(\mfr^b,M)$. Note that for $n\gg 0$,
$$\ell (M/(\mfr^b)^{n+1}M) = \ell (M/(\mfr^{(n+1)b}M) = e_0(\mfr,M)(n+1)b - e_1(\mfr,M).$$
This implies $ e_0(\mfr^b,M)= be_0(\mfr,M)$ and $ e_1(\mfr^b,M)= e_1(\mfr,M)$. 
Applying  Proposition \ref{H1} (\ref{EH11}) to the case $b=1$ we get 
$$e_1(\mfr,M) \le \binom{e_0(\mfr,M)}{2}.$$ 
(Of course, this inequality is known in \cite{K}.) 
If $e_0(\mfr,M) =1$, then the above inequality give $e_1(I,M) \le e_1(\mfr,M)=0$ and the statement trivially holds. Assume $e_0:= e_0(\mfr,M) \ge 2$. 
Since $b\ge 2$, we have $be_0-b \ge e_0$. Hence 
$$e_1(I,M) \le e_1(\mfr^b,M)  = e_1(\mfr,M) \le \binom{e_0}{2} \le \binom{be_0 - b}{2} = \binom{e_0(I,M)-b}{2}.$$
\end{proof}

\begin{rem} \label{regular}
Let us examine  when the bound in  Proposition \ref{H1} (\ref{EH11}) holds in the case $\dim M= 1$. 
Let $b$ be the largest positive integer  such that $IM \subseteq \mfr^bM$. 
Note that 
$b\le \ell(M/\mfr^bM) \le \ell(M/IM)\le e_0(I,M)$, and $e_0(I,M) \ge e_0(\mfr^b, M) = b e_0(\mfr,M)$. 
If $e_0(I,M) = b$, then $\ell(M/\mfr M) =1, \ e_0(\mfr,M) =1$, and from the inequality 
 Proposition \ref{H1} (\ref{EH11}) we get $e_1(I,M) \le 0$, whence $e_1(I,M) =0$. 
Then  Proposition \ref{H1}(\ref{EH11})  becomes an equality. 
Replacing $A$ by $A/\text{Ann}(M)$, one can now conclude  that $e_0(I,M) = b$ if and only if $M \cong A$,  $A$ is a regular ring and $I = (x^b)$, where $\mfr = (x)$. 
Hence we can exclude this case in further investigation.
\end{rem}

\begin{lem} \label{H3a} 
Let  $M$ be a  Cohen-Macaulay module of positive dimension $d$ and $I$ an $\mfr$-primary ideal. 
Let $b$ be the largest positive integer  such that $IM \subseteq \mfr^bM$. 
Assume that $e_0(I,M) > b$ and 
$$e_1(I,M) = \binom{e_0(I,M) -b +1}{2} + b - \ell(M/IM).$$
Then $b=1$.
\end{lem}
  \begin{proof}  Assume that $b\ge 2$. 
 First assume that $d=1$. 
 We have 
 
$$ \begin{array}{ll} e_1(I,M) & =  \binom{e_0(I,M) - b +1}{2}+b-\ell(M/IM)  \\ \\
 & =  \binom{e_0(I,M) - b}{2}+ e_0(I,M) -\ell(M/IM).
 \end{array}$$
 
\bigskip
\noindent
 Since $e_0(I,M) \ge \ell(M/IM)$, the above equality together with the inequality in Proposition \ref{H2} imply that  
 
 \begin{equation}\label{EH3a1}
 e_1(I,M) = \binom{e_0(I,M) - b}{2},
\end{equation}
and $e_0(I,M) = \ell(M/IM)$. 
Since $M$ is an one-dimensional Cohen-Macaulay module, $e_0(I,M) = \ell(M/xM)$ for some $x\in I$. 
This implies $IM = xM$, i.e. we can assume that $I$ is a parameter ideal. 
Then $e_1(I,M) = 0$,  and by (\ref{EH3a1}), $e_0(I,M) \le b+1$. 
By the assumption, we get $e_0(I,M) = b+1$.

Since $b+1 = e_0(I,M) \ge e_0(\mfr^b,M) = b e_0(\mfr, M)$ and $b\ge 2$, we  can conclude  that $e_0(\mfr,M) =1$. Let $y\in \mfr$ such that $e_0(\mfr, M) = \ell(M/ yM)$. 
Note that $ \ell(M/ yM) \ge \ell(M/\mfr M) = \mu(M)$. Hence,  then we must have $yM = \mfr M$ and $M$ is generated by one element, say $M = Au$.  
Replacing $A$ by $A/\text{Ann}(M)$, we may assume that $M= A$. 
Then $A$ is a regular ring, $\mfr = (y)$, see Remark \ref{regular}.
Since $x \in \mfr^b$ and $b$ is the largest  number satisfying this property, it implies that  $x= r y^b$ for some unit $r$. But then $e_0(I,M) = e_0(x,A) = b$, a contradiction. 
Hence, the assumption $b\ge 2$ is wrong and then $b=1$.

Now assume that $d \ge 2$. Let $x_1,...,x_{d-1}$ be an $M$-superficial sequence for $I$. 
Let $N= M/(x_1,...,x_{d-1})M$. 
Then $\dim N=1$ and 
$$e_1(I,N) = \binom{e_0(I,N) -b +1}{2} + b - \ell(N/IN).$$
Since  $IN \subseteq \mfr^b N$,  we must have $b=1$.
\end{proof}

Below are some characterizations for  the equality in (\ref{EH11}).
 
 \begin{prop} \label{H3}  
 Let  $M$ be an one-dimensional  Cohen-Macaulay $A$-module  and $I$ an $\mfr$-primary ideal. 
 Let $b$ be the largest positive integer  such that $IM \subseteq \mfr^bM$. 
 Assume that $e_0(I,M) \ge b+2$. 
 Then the following conditions are equivalent:
 \begin{itemize}
\item[(i)]  $e_1(I,M) = \binom{e_0(I,M)  - b+1}{2} + b - \ell(M/IM),$
\item[(ii)]  $P_{I,M}(z) = \frac{\ell(M/IM)+(b+1-\ell(M/IM))z+\sum_{i=2}^{e_0(I,M) -b}z^i}{1-z},$
\item[(iii)]  $a_0(G_I(M)) \le 0$ and $\reg(G_I(M)) = e_0(I,M) -b$,
 \item[(iv)] $\reg(G_I(M)) = \binom{e_0(I,M) -b +2}{2} +b  -e_1(I,M) - \ell(M/IM) -1$.
\end{itemize}
 If one of the above  conditions is satisfied, then $b=1$ and  $e_0(I,M) = e_0(\mfr, M)$.
\end{prop}
\begin{proof} 
For simplicity, in this proof we set $e_0:=e_0(I,M),\ e_1:=e_1(I,M)$ and $p:=pn(I,M)$. 

\noindent
(ii) $\Longrightarrow$ (i) is immediate from (\ref{EP1}).

\noindent
(i) $\Longrightarrow$ (ii)  Assume that  $e_1 = \binom{e_0 -b +1}{2} +1  - \ell(M/IM)$.  
Since $e_0\ge b+2$, by Proposition \ref{H1}(ii) and (iii), $p = e_0-b$, $H_{I,M}(n) = i+b$ for all $1 \le i \le e_0-b-1$ and $H_{I,M}(n) = e_0$ for all $n\ge e_0-b$. 
Substituting these values into the definition  (\ref{ser}) of the Hilbert series we then get (ii).

\noindent
(i) $\Longrightarrow$ (iii) By Proposition \ref{H1}(i), $a_0(G_I(M)) \le 0$. 
Since $e_0\ge b+2 $, by Proposition \ref{H1}(ii),  $\reg(G_I(M)) =  e_0 -b$.

\noindent
(iii) $\Longrightarrow$ (ii) By Lemma \ref{I6}(i), we have $p = \reg(G_I(M)) = e_0- b$. 
Since $a_0(G_I(M))\le 0$, 
$$\ell(I^tM/I^{t+1}M) = H_{G_I(M)}(t) = H_{\overline{G_I(M)}}(t) \ \text{for all}\ t\ge 1.$$
On the other hand, by Lemma \ref{I6}(iii), 
$$H_{ \overline{G_I(M)}}(t) =
\begin{cases}
t+b& \text{if} \  0\le  t \leq  e_0 - b-1 ,\\
e_0  & \text{if} \ t \ge  e_0- b .
\end{cases}$$

\noindent
Hence
$$\begin{array}{ll}
P_{I,M}(z)  & = \ell (M/IM) + \sum_{t=1}^{e_0-b-1} (t+b) z^t + \sum_{t\ge e_0-b} e_0z^t \\ \\
& = \frac{\ell(M/IM)+(b+1 -\ell(M/IM))z+\sum_{i=2}^{e_0 -b}z^i}{1-z}.
\end{array}$$

\noindent
(i) $\Longrightarrow$ (iv) Using (i) $\Leftrightarrow$ (iii), we have
$$\begin{array}{ll} 
\reg(G_I(M)) & = e_0 - b\\ \\
& = \binom{e_0-b+2}{2} - \binom{e_0-b +1}{2} +  e_1 + \ell(M/IM) - b -  e_1 - \ell(M/IM) + b - 1 \\ \\
& = \binom{e_0-b+2}{2}+b  -  e_1 - \ell(M/IM) -1.
\end{array}$$

\noindent
(iv) $\Longrightarrow$ (i) By \cite[Proposition 2.1]{D}, $\reg(G_I(M)) \le e_0 -b$. Therefore 
$$\begin{array}{ll}
\reg(G_I(M)) &\le e_0 - b\\ \\
&  = \binom{e_0-b+2}{2} - \binom{e_0-b +1}{2} +  e_1 + \ell(M/IM) - b -  e_1 - \ell(M/IM) + b - 1 \\ \\
&\le \binom{e_0-b+2}{2}+b  -  e_1 - \ell(M/IM) -1
 \ \ \text{(by Proposition \ref{H1})}.
\end{array}$$
By virtue of (iv), this implies  $e_1 = \binom{e_0}{2} +  1-  \ell(M/IM)$.

Finally assume (i). 
By Lemma \ref{H3a}, $b=1$. Further,  since $e_0-1 \ge 2$, we have   
$\binom{e_0-1}{2} +1 \le \binom{e_0}{2}$. 
Hence,  if $e_0 \neq e_0(\mfr,M)$, by virtue of Proposition \ref{H1}(\ref{EH12}), we cannot have (i), a contradiction.
\end{proof}
\begin{exm} \label{H4} 
Let $A= k[[t^2,t^3]]$ and $M= \mfr$. 
Then $e_0(\mfr,M) = e_0(\mfr, A) = 2$, while $e_1(\mfr, A) =1$ and $e_1(\mfr, M) = 0$. 
Hence, the condition (i) of Proposition \ref{H3} is satisfied for both pairs $(\mfr, A)$ and $(\mfr,M)$.
Note that $\ell(\mfr^n/ \mfr^{n+1}) = 2 $ for all $n\ge 1$. 
So, $pn(\mfr, A) = 1$ and $pn(\mfr, M) = 0$. 
The ring $G_{\mfr}(A)$ and the module $G_I(M)$ are Cohen-Macaulay, but $\reg(G_{\mfr}(A) )= 1 =  e_0(\mfr, A) -1$, while $\reg (G_{\mfr}(M) )= 0 < e_0(\mfr,M) -1 = 1$.  
This shows that no of the conditions (ii), (iii) and  (iv) in Proposition \ref{H3}  holds.  
So the condition $e_0(I,M) \ge b+2$ in  Proposition \ref{H3} cannot be omitted.
\end{exm}

Using (iv) and (v) of Lemma \ref{I6} and  (i'), (ii') and (iii') of Proposition \ref{H1}, similar arguments of the proof  of Proposition \ref{H3} give:

\medskip
 \begin{prop} \label{H5}  
 Let  $M$ be an one-dimensional  Cohen-Macaulay $A$-module  and $I$ an $\mfr$-primary ideal such that $I \subseteq \mfr^b$, $ e_0(I,M) > e_0(\mfr^b, M)$ and $ e_0(I,M) \ge b+3$, where $b$ is a positive integer. 
 Then the following conditions are equivalent:
   \begin{itemize}
 \item[(i)]  $e_1 = \binom{e_0(I,M) -b}{2} +b +1  - \ell(M/IM)$,
 \item[(ii)]  $P_{I,M}(z) =\frac{\ell(M/IM) + (b+2 -\ell(M/IM))z+\sum_{i=2}^{e_0(I,M) - b-1}z^i}{1-z}$,
 \item[(iii)]  $a_0(G_I(M)) \le 0$ and $\reg(G_I(M)) = e_0(I,M) -b-1$,
  \item[(iv)] $\reg(G_I(M)) = \binom{e_0(I,M) -b +1}{2} +b  -e_1(I,M) - \ell(M/IM)$.
    \end{itemize}
\end{prop}

%%%%%%%%%%%%%%%%%%%%%%%%%%%%%%%%%%%%%%%%%%%%%%%%%%%%%%%%%%%%%%%%%%%%%%%%%%%%%%%%%%%%%%%%%%%%%%%%%%%
%%%%%%%%%%%%%%%%%%%%%%%%%%%%%%%%%%%%%%%%%%%%%%%%%%%%%%%%%%%%%%%%%%%%%%%%%%%%%%%%%%%%%%%%%%%%%%%%%%%
%%%%%%%%%%%%%%%%%%%%%%%%%%%%%%%%%%%%%%%%%%%%%%%%%%%%%%%%%%%%%%%%%%%%%%%%%%%%%%%%%%%%%%%%%%%%%%%%%%%
%%%%%%%%%%%%%%%%%%%%%%%%%%%%%%%%%%%%%%%%%%%%%%%%%%%%%%%%%%%%%%%%%%%%%%%%%%%%%%%%%%%%%%%%%%%%%%%%%%%
%%%%%%%%%%%%%%%%%%%%%%%%%%%%%%%%%%%%%%%%%%%%%%%%%%%%%%%%%%%%%%%%%%%%%%%%%%%%%%%%%%%%%%%%%%%%%%%%%%%
%%%%%%%%%%%%%%%%%%%%%%%%%%%%%%%%%%%%%%%%%%%%%%%%%%%%%%%%%%%%%%%%%%%%%%%%%%%%%%%%%%%%%%%%%%%%%%%%%%%
\medskip
\section{The first Hilbert function of an $\mfr$-primary ideal} \label{e1Id}

In this section we consider the case $M=A$, that is we study the first Hilbert coefficient of an 
$\mfr$-primary ideal $I$ of a Cohen-Macaulay local ring $(A,\mfr)$. 
If $b\ge 2$, see Theorem \ref{EB3} below.
If $b=1$, then the Rossi-Valla bound in the statement (i) of the following lemma is clearly much better. 

\begin{lem} \label{T1} Let $(A,\mfr)$ be a $d$-dimensional Cohen-Macaulay ring and $I$ an $\mfr$-primary ideal. Then
\begin{itemize}
\item[(i)] {\rm (\cite[Theorem 3.2]{RV1}) }
\begin{equation}\label{ET10}
e_1(I) \le \binom{e_0(I)}{2} - \binom{\mu (I) - d}{2} - \ell(A/I) +1,
\end{equation} 
where $\mu(I)$ denotes the number of generators of $I$.
\item[(ii)] {\rm (A partial case of \cite[Theorem 3.2]{RV1}) } If $d=1$, then we also have
$$e_1(I) \le \binom{e_0(I)}{2} - \binom{\mu (\tilde{I}) - 1}{2} - \ell(A/\tilde{I}) +1.$$
\end{itemize} 
\end{lem}
\begin{proof} There is an unclear step in the proof of \cite[Theorem 3.2]{RV1} in the case $d=1$: from the context, $\lambda$ in \cite[(8)]{RV1} should be $\ell(A/\tilde{I})$, see at the beginning of \cite[Section 3]{RV1}. 
Therefore we give here a correction of this part. 
So, we may assume that $d=1$ and we need to show 
\begin{equation}\label{ET11}
e_1(I) \le \binom{e_0(I)}{2} - \binom{\mu (I) - 1}{2} - \ell(A/I) +1.
\end{equation}
If $\mu(I) = 1$, then $I$ is a parameter ideal, which implies $e_1(I) =0$ and the inequality holds true. Now let $\mu(I)\ge 2$. By \cite[Theorem 3.1]{RV1},
\begin{equation}\label{ET12}
e_1(I) \le \binom{e_0(I)}{2} - \binom{g - 1}{2} - \ell(A/\tilde{I}) +1,
\end{equation}
where
$$g = \ell(\tilde{I}/ \widetilde{I^2}) + \sum_{i\ge 2}\ell\left(\frac{\widetilde{I^{i+1}}}{I\widetilde{I^i} + \widetilde{I^{i+2}}}\right).$$
At the end of the proof of \cite[Theorem 3.1]{RV1}, it is shown that
$$\begin{array}{ll}
g &\ge \ell(\tilde{I}/ \widetilde{I^2} + I\mfr) + \ell(\widetilde{I^2} + I\mfr/I\mfr )\\ \\
& = \ell(\tilde{I}/ I\mfr) = \ell(\tilde{I}/I) + \mu(I).
\end{array}$$
Set $\tilde{l} := \ell(\tilde{I}/I) $. 
Then we get

\begin{eqnarray}\nonumber
\binom{g - 1}{2} + \ell(A/\tilde{I}) &\ge &\binom{\mu(I) - 1 + \tilde{l}}{ 2} +
 \ell(A/\tilde{I})\\
&=& \binom{\mu(I) -1}{2} + \frac{(2\mu(I) - 3) \tilde{l} +  \tilde{l}^2}{2} + 
\ell(A/\tilde{I}) \nonumber \\
& \ge & \binom{\mu(I) -1}{2} + \frac{ \tilde{l}( \tilde{l}+1)}{2} + 
\ell(A/\tilde{I}) \ \text{(since} \ \mu(I) \ge 2). \label{ET13}
\end{eqnarray}

If $\tilde{l} = 0$, then $\ell(A/\tilde{I}) = \ell(A/I)$. If $\tilde{l}\ge 1$, then
\begin{equation}\label{ET15}
\frac{ \tilde{l}( \tilde{l}+1)}{2} +  \ell(A/\tilde{I}) \ge \tilde{l} + \ell(A/\tilde{I}) = \ell(A/I).
\end{equation}
In both cases, from (\ref{ET13}) we get
\begin{equation}\label{ET16} 
\binom{g - 1}{2} + \ell(A/\tilde{I})  \ge \binom{\mu(I) -1}{2}  +\ell(A/I).
\end{equation}
Combining this with (\ref{ET12}) we immediately get (\ref{ET11}).
\end{proof}

\medskip
Using the Rossi-Valla bound (\ref{ET10}) we can immediately see that if $I$ satisfies the condition (i)  of Proposition \ref{H3} (with $M=A$), then $\mu(I) \le 2$. 
However, if $I$ is a parameter ideal, then $e_1(I) =0$, while 
$\binom{e_0(I)}{2} + 1 - \ell(A/I) = 
\binom{e_0(I)}{2} + 1 -  e_0(I) \ge 1$, 
provided  $e_0(I) \ge 3$. 
This contradicts the condition (i). So,  $\mu(I) =2$. 
Below are more information on the structure of $I$ and $A$ itself, when $I$ satisfies the condition (i)  of Proposition \ref{H3}, or equivalently, when the Rossi-Valla bound (\ref{ET10}) is attained, provided $\mu = 2$.

\begin{prop} \label{T2} Let  $(A,\mfr)$ be an one-dimensional  Cohen-Macaulay ring and $I$ an $\mfr$-primary ideal such that $ e_0(I) \ge 3$ and 
$e_1(I) = \binom{e_0(I)}{2} + 1 - \ell(A/I)$. Then we have
\begin{itemize}
\item[(i)] $\Itil = \mfr$, $I^2 = \mfr I$ and $\mu(I) = 2$.
\item[(ii)] $\mu(\mfr) \in \{2, 3\}$,  and 
\item[(iii)] If $\mu(\mfr) =2$, then $I=\mfr$ and  $G(\mfr)$ is a Cohen-Macaulay ring.
\item[(iv)] If $\mu(\mfr)=3$, then $I^n=\mfr^n$ for all $n\ge 2$ and  $\ell(A/I) = 2$. 
In this case $\depth G(I) = 0$.
\end{itemize} 
\end{prop}
\begin{proof} We set $e_0 := e_0(I)$.
\noindent
(i)  $\mu(I) = 2$ was shown above. By Proposition \ref{H1}(ii), $b=1$ and $pn(I) = e_0(I) -1$. 
Hence, by Lemma \ref{I6}(iii), $\ell(\overline{G(I)_0}) = 1$. By (\ref{EP3}), we then get $\ell(A/\Itil) = \ell(\overline{G(I)_0}) = 1$. Since $\Itil \subseteq \mfr$, we must have $\Itil = \mfr$.
By Proposition \ref{H1}(iii), $\ell(I/I^2) = 2$. Since $2 = \mu(I) = \ell(I/\mfr I) \le \ell(I/I^2) = 2$, we get $I^2 = \mfr I$. 

\noindent
(ii) Since $\mu(I) =2$, the equality in (\ref{ET10})  also holds for $I$. From (\ref{ET13}), (\ref{ET15}) and (\ref{ET16}) we must have $\ell(\Itil/I) \leq 1$. Since 
$\Itil = \mfr$, we get $\ell(A/I)\le 2$, and by Lemma \ref{T1}(ii), we now have
$$\binom{e_0}{2} + 1 - \ell(A/I) = e_1 \le \binom{e_0}{2} - \binom{\mu(\mfr) - 1}{2} .$$
This implies $\binom{\mu(\mfr) - 1}{2} \le 1$, whence $\mu(\mfr) \le 3$. By Proposition \ref{H3}, $e_0(\mfr) = e_0(I) \ge 3$. Hence $\mu(\mfr) \ge 2$.

Assume that $a$ is an element in a minimal basis of $I$. We first show that $a\not\in\mfr^2$. Assume by contrary, that $a \in \mfr^2$. Since $\Itil = \mfr$, we get
$$a \in \mfr^2 \cap I = (\Itil)^2 \cap I \subseteq \widetilde{I^2} \cap I.$$
By Proposition \ref{H3}(iii) and (\ref{EP3}), we get
$$0 = H^0_{G(I)_+}(G(I))_1 \cong \frac{\widetilde{I^2} \cap I}{I^2},$$
which implies 
\begin{equation}\label{ET21}\widetilde{I^2} \cap I =I^2.
\end{equation}
Hence $a\in I^2$, a contradiction.

The condition  that any element in a minimal basis of $I$ does not belong to $\mfr^2$  implies that the images of $a_1,a_2$ of a minimal basis of $I$ in $\mfr/\mfr^2$ are linearly independent. 
This means that $\{a_1,a_2\}$ is a part of minimal basis of $\mfr$.

\noindent
(iii) If $\mu(\mfr) = 2$ then $I = \mfr$ and
 $H^0_{G_+}(G(I))_0 =0$.
Since $a_0(G(I))\le 0$ (by Proposition \ref{H3}(iii)), $H^0_{G_+}(G(I)) = 0$, and $G(\mfr) = G(I)$ is a Cohen-Macaulay ring.

\noindent
(iv) Assume now that $\mu(\mfr) = 3$. As shown above $\ell(A/I)\le 2$. So we must have  $\ell(A/I)= 2$. Assume that $\mfr= (a_1,a_2,a_3)$, where $\{ a_1,a_2\}$ is a minimal basis of 
$I$. Moreover, we may assume that both elements $a_1,a_2$ are non-zero divisors of $A$. Since $\Itil = \mfr$, by (\ref{ET21}), we have
$$a_1 a_3 \in I \cap (\Itil)^2  \subseteq I \cap \widetilde{I^2} = I^2 = (a_1,a_2)^2.$$
Hence 
\begin{equation}\label{ET22}
a_1a_3 = y_1a_1^2 + y_2a_1 a_2 + y_3 a_2^2,
\end{equation}
for some $y_1,y_2,y_3 \in A$.  Replacing $a_3$ by $a_3- y_1a_1 - y_2a_2$ in the above relation, we may assume that
$$a_1a_3 = y_3a_2^2.$$
Analogously, we can find $z_1,z_2,z_3 \in A$ such that
\begin{equation} \label{ET23} a_2a_3 = z_1a_1^2 + z_2a_1 a_2 + z_3 a_2^2.\end{equation}
Then 
$$\begin{array}{ll}
a_2a_3^2 &= z_1a_1^2 a_3 + z_2a_1 a_2 a_3 + z_3 a_2^2 a_3\\
&= z_1 a_1 y_3a_2^2 + z_2 y_3 a_2^3 + z_3 a_2^2 a_3.
\end{array}$$
Since $a_2$ is a non-zero divisor, this implies
$$a_3^2 = z_1y_3  a_1 a_2 + z_2 y_3 a_2^2+ z_3 a_2 a_3 \in I^2.$$
Together with (\ref{ET22}) and (\ref{ET23}), this shows that $I^2 = \mfr^2$, and by induction $I^n = \mfr^n$ for all $n\ge 2$.
In this case, by (\ref{EP3}), $H^0_{G_+}(G(I))_0 \cong \frac{\Itil}{I} = \frac{\mfr}{I} \neq 0$, $\depth G(I) = 0$.
\end{proof}

\medskip
\begin{rem}
The Cohen-Macaulayness of $G(\mfr)$ in (iii) of the above proposition is known long time ago, see, e.g. \cite[p. 19]{Sa}.

If $\mu (I) >2$, then  the Rossi-Valla bound (\ref{ET10}) is much better. An ideal, for which   the Rossi-Valla bound (\ref{ET10}) is attained, may have an arbitrary number of generators. For an example, take $I = \mfr$ in $A= k[[t^a,t^{a+1},...,t^{2a-1}]], a\ge 3$. 
Then $e_0(\mfr) = a$ and 
$$e_1(\mfr) = a-1 = \binom{e_0(\mfr)}{2} - \binom{\mu(\mfr) -1}{2}.$$
If $I= \mfr$  the Rossi-Valla bound (\ref{ET10})   is Elias' bound given in \cite[Theorem 1.6]{E0}. 
In \cite[Theorem 3.1]{ERV1}, there is a characterization in terms of Hilbert series for an one-dimensional Cohen-Macaulay  ring such that the Elias' bound is attained. See also \cite[Proposition 3.3]{RV1} for a shorter proof.
\end{rem}

\begin{exm} \label{T3} Let $a\ge 3$ and $A= k[[t^a, t^{a+1}, t^{a^2-a-1}]]$ and $I= (t^a,t^{a+1})$. Then 
$$\ell(A/I) = 2,\ e_0(I) = e_0(\mfr) = a,\ e_1(I) = e_1(\mfr) = \binom{a}{2} -1 = 
\binom{e_0(I)}{2} + 1 - \ell(A/I).$$
This is the situation in (iv) of the above proposition. 
Note that $G(\mfr)$ is  a Cohen-Macaulay ring only in the case $a=3$. 
This was indicated in \cite[p. 19]{Sa} in the case $a=3$ and in \cite[Proposition 4.6(2)]{E0} for $a\ge 4$.
\end{exm}
 
We now give a new bound on $e_1(I)$ for an $\mfr$-primary ideal $I\subseteq \mfr^b$ and $b\ge 2$. It is in fact a correction of the bound given in \cite[Proposition 1.1]{E2}. The following result was stated for any dimension $d\ge 1$, but the proof there is  valid only for $d=1$, because in general one cannot find an element $x\in \mfr$ such that it is simultaneously superficial for both $\mfr$ and $I$.

\begin{lem}\cite[Proposition 1.1]{E2} \label{EB1} Let $I\subseteq \mfr^b$ be an $\mfr$-primary ideal of an one-dimensional Cohen-Macaulay ring $A$. Then
$$e_1(I) \le (e_0(\mfr) - 1)(e_0(I) - be_0(\mfr)) + e_1(\mfr).$$
\end{lem}

Modifying the bound in the above lemma, we can give a new bound on $e_1(I)$ for any dimension.

\begin{thm} \label{EB3} Let $A$ be a Cohen-Macaulay ring of dimension $d\ge 1$. Let $I\subseteq \mfr^b$ be an $\mfr$-primary ideal, where $b\ge 1$. 
Then
$$e_1(I) \le \frac{1}{2b-1}\binom{e_0(I) - b+1}{2} - \binom{\mu(m) - d}{2}.$$
\end{thm}

\begin{proof} First consider the case $d=1$. By Lemma \ref{EB1},
$$e_1(I) \le (e_0 - 1)(e_0(I) - be_0) + e_1(\mfr),$$
where we set $e_i := e_i(\mfr)$. By \cite[Theorem 1.6]{E0},
$$e_1 \le \binom{e_0}{ 2} - \binom{\mu(m) - 1}{2}.$$
Hence
$$\begin{array}{ll}
e_1(I) &\le (e_0 - 1)(e_0(I) - be_0)  + \frac{e_0(e_0-1)}{2} - \binom{\mu(m) - 1}{2}\\ \\
&= e_0^2(-b+\frac{1}{2}) + e_0(e_0(I) + b-\frac{1}{2}) - e_0(I)  - \binom{\mu(m) - 1}{2}.
\end{array}$$

\noindent
The function
$$f(t) = (-b+\frac{1}{2}) t^2 + (e_0(I) + b-\frac{1}{2}) t - e_0(I)$$
reaches its maximum at $t_0 = \frac{e_0(I) + b-\frac{1}{2}}{2(b-\frac{1}{2})}$ and
$$f(t_0) = \frac{(2e_0(I) - 2b+1)^2}{8(2b-1)} = \frac{1}{2b-1} \left\lbrace  \binom{e_0(I) - b+1}{2} + \frac{1}{8}\right\rbrace  .$$
 Note that $\lfloor \frac{m+\alpha}{n} \rfloor = \lfloor \frac{m}{n} \rfloor $ for any integers $n\ge 1$, $m$ and a real number $0\le \alpha <1$. 
 Hence
 $$ \begin{array}{ll}
 e_1(I) & \le \lfloor f(t_0)\rfloor - \binom{\mu(m) - 1}{2} \\  \\
 &= \lfloor \frac{1}{2b-1} \binom{e_0(I) - b+1}{2} \rfloor  - \binom{\mu(m) - 1}{2} \\ \\
 &\le  \frac{1}{2b-1} \binom{e_0(I) - b+1}{2}  - \binom{\mu(m) - 1}{2} .
\end{array}$$
Now let $d\ge 2$. Let $x\in I$ be a superficial element. 
Then $e_0(I/x) = e_0(I), \ e_1(I/x) = e_1(I)$, $I/x \subseteq (\mfr/x)^b$ and $\mu (\mfr/x) \ge \mu(\mfr)-1$. 
Hence, the conclusion follows by induction on the dimension.
\end{proof}
 
 \begin{rem} \label{EB4} Let $b\ge 2$. Then Theorem \ref{EB3} gives
\begin{equation}\label{EEB41} 
e_1(I) \le \lfloor \frac{1}{3}\binom{e_0(I) - b+1}{2}\rfloor - \binom{\mu(m) - d}{2}.
\end{equation}
 It is easy to check that 
 $$\lfloor \frac{1}{3}\binom{e_0(I) - b+1}{2}\rfloor  \le \binom{e_0(I) - b}{2}.$$
 This give another proof of Corollary \ref{H2} in the case $M=A$. 
 It  is clearly better than the bound of  Corollary \ref{H2} if $\mfr$ is generated by at least $d+2$ elements. 
 If $e_0(I) \ge b+5$, then from (\ref{EEB41}) we get a much better bound:
 $$e_1(I) \le \frac{1}{2}\binom{e_0(I) - b}{2} - \binom{\mu(m) - d}{2}.$$
\end{rem}

\begin{rem} \label{EB5} We now give a brief account of upper bounds on $e_1(I)$ of an $\mfr$-primary ideal $I$ of an one-dimensional Cohen-Macaulay ring $A$ such that $I \subsetneq \mfr^2$.
\begin{itemize}
\item[(i)] The first Rossi-Valla  bound (\ref{ET10}):
$$e_1(I) \le \binom{e_0(I)}{2} - \binom{\mu (I) - 1}{2} - \ell(A/I) +1.$$
\item[(ii)] The second  Rossi-Valla  bound (see Lemma \ref{T1}(ii)):
$$e_1(I) \le \binom{e_0(I)}{2} - \binom{\mu (\tilde{I}) - 1}{2} - \ell(A/\tilde{I}) +1.$$
\item[(iii)] Elias' bound (see Proposition \ref{EB1}):
$$e_1(I) \le (e_0(\mfr) - 1)(e_1(I) - be_0(\mfr)) + e_1(\mfr).$$
\item[(iv)] The Hanumanthu-Huneke bound \cite[Corollary 2.9]{HH}: Under the additional condition that $A$ is an analytically unramified local domain with algebraically closed residue field, we have
$$e_1(I) \le \binom{e_0(I) - \ell(A/\bar{I}) + 1}{2},$$
where $\bar{I}$ denotes the integral closure of $I$.
\item[(v)]  The case $b=1$ of Theorem \ref{EB3}
$$e_1(I) \le \binom{e_0(I)}{2} - \binom{\mu(m) -  1}{2}.$$
\end{itemize} 
One can give examples to show that these bounds are independent. Note that the bounds in (i) and (v) can be lifted to higher dimensions, while we could not do the same for the other bounds.
\end{rem}

%%%%%%%%%%%%%%%%%%%%%%%%%%%%%%%%%%%%%%%%%%%%%%%%%%%%%%%%%%%%%%%%%%%%%%%%%%%%%%%%%%%%%%%%%%%%%%%%%%%
%%%%%%%%%%%%%%%%%%%%%%%%%%%%%%%%%%%%%%%%%%%%%%%%%%%%%%%%%%%%%%%%%%%%%%%%%%%%%%%%%%%%%%%%%%%%%%%%%%%
%%%%%%%%%%%%%%%%%%%%%%%%%%%%%%%%%%%%%%%%%%%%%%%%%%%%%%%%%%%%%%%%%%%%%%%%%%%%%%%%%%%%%%%%%%%%%%%%%%%
%%%%%%%%%%%%%%%%%%%%%%%%%%%%%%%%%%%%%%%%%%%%%%%%%%%%%%%%%%%%%%%%%%%%%%%%%%%%%%%%%%%%%%%%%%%%%%%%%%%
%%%%%%%%%%%%%%%%%%%%%%%%%%%%%%%%%%%%%%%%%%%%%%%%%%%%%%%%%%%%%%%%%%%%%%%%%%%%%%%%%%%%%%%%%%%%%%%%%%%
%%%%%%%%%%%%%%%%%%%%%%%%%%%%%%%%%%%%%%%%%%%%%%%%%%%%%%%%%%%%%%%%%%%%%%%%%%%%%%%%%%%%%%%%%%%%%%%%%%%
\section{The second Hilbert coefficient} \label{e2Mo}

 Rhodes  \cite[Proposition 6.1(iv)]{Rh} proved that $e_2(I,M) \le \binom{e_1 (I,M)}{2}$.  Combining with the bound in Proposition \ref{H1}, we get $e_2(I,M) < \frac{1}{8}e_0(I,M)^4$.  
 In the case $I= \mfr$ and $M= A$,   there is a much better bound given in  \cite[Theorem 2.3]{ERV1}. 
 The bound also involves $e_1(\mfr)$ and some rather technical invariants.  
 As a consequence, it was shown there that 
 $e_2(\mfr) \le \binom{e_1(\mfr)}{2} - \binom{\mu(\mfr) - d}{2}$, which is of course better than Rhodes' bound in the case $I=\mfr$. 
 Applying known bounds on $e_1(\mfr)$ to the bound in \cite[Theorem 2.3]{ERV1}, one can show that $e_2(\mfr) < \frac{2}{3} e_0(\mfr)^3$.
 
The aim of this section is to give a new bound on $e_2(I,M)$ in terms of $e_0(I,M)$, which is less than  $\frac{1}{6}e_0(I,M)^3$, and to characterize when this bound is attained. 
In the case $M=A$, after finding some  relationships between the reduction number and the Hilbert coefficients, using Theorem \ref{EB3} we can give a better bound for a large class of $I$, see Theorem \ref{AB8}.

\begin{thm}\label{B1} Let  $M$ be a  Cohen-Macaulay module of $\dim(M)= d \geq 2$ over $(A,\mfr)$. 
Let $I$ be an $\mfr$-primary ideal such that $IM \subseteq \mfr^b M$, where $b$ is a positive integer.  Then
$$ e_2(I,M)  \leq \binom{e_0(I,M) -b+1}{3} .$$
\end{thm}

\begin{proof}
By standard technique  we may assume that $d=2$.

Let $x \in I \backslash I^2$ be an $M$-superficial element for $I$.  
Let  $N:= M/xM$. 
By Lemma \ref{I2}(ii) and (iii),  $e_i(I,M) = e_i(I,N)$ for  $i = 0, 1$. 
For short, we write  $p := pn(I,N)$ and $e_0 := e_0(I,M) = e_0(I,N)$. 
Then
\begin{equation} \label{EB11}
P_{I,N}(z)   = \frac{H_{I,N}(0)+\sum_{i=0}^{p-1}(H_{I,N}(i)-H_{I,N}(i-1))z^i+(e_0 -H_{I,N}(p-1))z^{p}}{1-z}.
\end{equation}
By (\ref{EP1}), we have
\begin{eqnarray}
e_2(I,N) & = & \frac{\sum_{i=0}^{p-1}(i-1)i(H_{I,N}(i)-H_{I,N}(i-1))+p(p-1)(e_0-H_{I,N}(p - 1))}{2!}   \nonumber\\
& = & -\sum_{i=1}^{p-1}i H_{I,N}(i) +\frac{p(p-1)}{2}e_0  \nonumber  \\
& = & \sum_{i=1}^{p-1}  i(e_0-H_{I,N}(i))   \nonumber  \\
& \le & \sum_{i=1}^{e_0-b -1} i (e_0-H_{I,N}(i))  \ \  \text{(by  Lemma \ref{I6}(iii)) } \label{EB12}
\end{eqnarray}
\begin{eqnarray}
& \le  & \sum_{i=1}^{e_0-b -1} i (e_0- i -b - \ell (H^0_{G(I)_+}(G_I(N))_i)  \ \ \text{(by Lemma \ref{I6}(ii))}   \nonumber \\
& \le &  \sum_{i=1}^{e_0-b -1} i(e_0-i-b)  
=  \binom{e_0-b+1}{3} .   \label{EB13} 
\end{eqnarray}
Since $M$ is a Cohen-Macaulay module, by Lemma \ref{I2}(iv), 
\begin{equation}\label{EB14} e_2(I,N) = e_2(I,M) + \sum_{i=0}^n \ell\left(\frac{I^{i+1}M : x }{I^iM}\right) \ge e_2(I,M). 
\end{equation}
Hence the inequality (\ref{EB13}) gives $e_2(I,M) \le \binom{e_0-b+1}{3}$.
\end{proof}

\medskip
\begin{rem}
Assume that $IM \subseteq \mfr^b M$. If $e_0(I,M) \le b+1$, then by the above theorem, we get $e_2(I,M) \le 0$. 
From the famous result of Narita \cite{Na} on the non negativity of the second Hilbert coefficient (see \cite[Proposition 3.1]{RV2} for a short proof in the module case), this implies  $e_2(I,M) = 0$.
Hence we can omit this case when dealing with the border case of the above theorem. 
The following result say that if  the above bound is attained, then $b=1$ and $I$ satisfies the conditions of Proposition \ref{H3}.
\end{rem}

\begin{thm}\label{B2} 
Let  $M$ be a  Cohen-Macaulay module of $\dim(M)= d \ge 2$ over $(A,\mfr)$ and $I$  an $\mfr$-primary ideal.
Let $b$ be the largest integer such that $IM \subseteq \mfr^b M$.
Assume that $e_0(I,M) \ge b+2$. The following conditions are equivalent:
  \begin{itemize}
 \item[(i)]   $e_2(I,M)  = \binom{e_0(I,M) - b +1}{3},$
 \item[(ii)]  $P_{I,M}(z) = \textstyle{\frac{\ell(M/IM)+(1+b-\ell(M/IM))z+\sum_{i=2}^{e_0(I,M)-b}z^i}{(1-z)^d}},$
 \item[(iii)]   $\depth(G_I(M)) \geq d-1$ and  $e_1(I,M)  = \binom{e_0(I,M) - b +1}{2} + b - \ell(M/IM)$,
  \item[(iv)]    $\depth(G_I(M)) \geq d-1$, $\reg(G_I(M)) = e_0(I)  - b$ and $a_{d-1}(G_I(M)) \le 1-d$,
 \item[(v)] $\depth(G_I(M)) \geq d-1$ and $\reg(G_I(M)) = \binom{e_0(I,M) - b +2}{2} +b  - e_1(I,M) -\ell(M/IM)-1$.
 \end{itemize}
 If one of the above conditions holds, then $b=1$.
\end{thm}
\begin{proof} For simplicity, we set $e_i := e_i(I,M),\ i\in \{ 0,1,2\} $ and $G:= G(I)$.
First, let $d=2$.

By (\ref{EP1}) it is clear that (ii) implies (i).
Assume (i), i.e. $e_2 = \binom{e_0 - b +1}{3}$.  
Let $x \in I \setminus I^2$ be an $M$-superficial element for $I$.  
Let $N:= M/xM$.   
By Lemma \ref{I2}, $e_i(I,N) = e_i(I,M) = e_i$ for $i= 0,1$, and by (\ref{EB14}), $e_2(I,M) \le e_2(I,N)$.  
Since $e_2(I,N) \le \binom{e_0 - b +1}{3}$ (see (\ref{EB13})),  we must have  $e_2(I,N) = e_2(I,M) = \binom{e_0 - b +1}{3}$. 
By Lemma \ref{I2}(v), the initial form $x^*\in I/I^2$ is a regular element on $G_I(M)$. This means $\depth G_I(M) >0$. 
Note that $G_I(N) \cong G_I(M)/x^*G_I(M)$.

Moreover, since $e_0 \ge b+2$, using (\ref{EB12}) and  (\ref{EB13}) we also have $p = e_0-b$, where  $p:=pn(I,N)$, and
$H_{I,N}(n) = n+b$ for  all $1\le  n\leq p$. By (\ref{EB11}) we then get
\begin{equation}\label{EB21}P_{I,N}(z) = \frac{\ell(N/IN)+(1+b-\ell(N/IN))z+\sum_{i=2}^{e_0-b}z^i}{1-z}.
\end{equation}  
Therefore, using Lemma \ref{I2}(v) again, we get (ii). Thus (i) $\Longleftrightarrow $ (ii) and they imply $\depth G_I(M) >0$.

\noindent
(ii) $\Longrightarrow $ (iii) The first part  $\depth G_I(M) >0$ was just shown, while the second part immediately follows from (\ref{EP1}). 

\noindent
(ii) $\Longrightarrow$  (iv) and (v) The first part  $\depth G_I(M) >0$ was  shown above. 
Since $x^*$ is a regular element on $G_I(M)$, by Lemma \ref{I2}(v), it implies that (\ref{EB21}) holds. This means $(I,N)$ satisfies the condition (ii) of Proposition \ref{H3}.
By the conditions (iii) and (iv) of Proposition \ref{H3}, we get
$$\begin{array}{ll}
\reg(G_I(N)) & = \binom{e_0 - b +2}{2} +b  - e_1 -\ell(N/IN)-1,\\
\reg(G_I(N)) & = e_0  - b\   \text{and}\  a_0(G_I(N)) \le 0.
\end{array}$$
Note that  $\reg G_I(M) = \reg G_I(M)/x^* G_I(M) = \reg G_I(N)$ and $\ell(M/IM) =\ell(N/IN)$. Hence 
$$\begin{array}{ll}
\reg(G_I(M)) & = \binom{e_0 - b +2}{2} +b  - e_1 -\ell(M/IM)-1,\\
\reg(G_I(M)) & = e_0  - b.
\end{array}$$
 Thus (v) is proved. Further, since $a_0(G_I(N)) \le 0$, from  the exact sequence
 $$0= H^0_{G_+} (G_I(N))_n \cong H^0_{G_+} (G_I(M)/ x^*G_I(M))_n \rightarrow H^1_{G_+} (G_I(M))_{n-1} \rightarrow  H^1_{G_+} (G_I(M))_n,$$
we get inclusions
$$H^1_{G_+} (G_I(M))_{n-1} \hookrightarrow  H^1_{G_+} (G_I(M))_n,$$
for all $n\ge 1$. This implies that $H^1_{G_+} (G_I(M))_n = 0$ for all $n\ge 0$, or equivalently,  $a_1(G_I(M)) \le - 1$. Summing up,  (ii) also implies (iv).

If one of the conditions (iii), (iv) and (v) is fulfilled,  then one of the conditions (i), (iii) or (iv) in Proposition \ref{H3} holds for the pair $(I,N)$. Hence by Proposition \ref{H3}(ii)
$$\begin{array}{ll}
P_{I,N}(z) & = \frac{\ell(N/IN)+(b+1-\ell(N/IN))z+\sum_{i=2}^{e_0-b}z^i}{(1-z)}
\\
& = \frac{\ell(M/IM)+(b+1-\ell(M/IM))z+\sum_{i=2}^{e_0-b}z^i}{(1-z)}.
\end{array}$$
Using Lemma \ref{I2}(v), we then get (ii). The proof of the case $d=2$ is completed.

Assume now $d>2$. Then (ii) $\Longrightarrow $ (i) follows from (\ref{EP1}). 

Assume (i). Let $x \in I \setminus I^2$ be an $M$-superficial element for $I$ and $N:= M/xM$.  Then $\dim N= d-1$ and the pair $(I,N)$ satisfies the condition (i). 
By induction hypothesis, $\depth G_I(N) \ge d-2$. 
Using Sally's descent  (see \cite[Lemma 2.2]{HM} or \cite[Lemma 1.4]{RV2}), we can deduce that $\depth G_I(M) \ge d-1$. 
This implies that $x^*$ is regular on $G_I(M)$. 
By Lemma \ref{I2}(v), (ii) follows. 
Further, we  have $\reg(G_I(M)) = \reg(G_I(N))$. 
Using the exact sequence
 $$ H^{d-2}_{G_+} (G_I(N))_n \cong H^{d-2}_{G_+} (G_I(M)/x^* G_I(M))_n \rightarrow H^{d-1}_{G_+} (G_I(M))_{n-1} \rightarrow H^{d-1}_{G_+} (G_I(M))_{n} ,$$
 one can see that $a_{d-2}(G_I(N)) \le 2-d$ implies $a_{d-1}(G_I(M)) \le 1-d$. 
 Since $(I,N)$ satisfies the condition (iii), (iv), (v), we then get that also $(I,M)$ satisfies these conditions.

 Conversely,  assume that $\depth G_I(M) \ge d-1$. 
 Then, by Sally's descent,  we get
 $\depth(G_I(N)) \ge d-2$ and $x^*$ is regular on $G_I(M)$. Hence, we have the following  exact sequence
 $$0 \rightarrow H^{d-2}_{G_+} (G_I(N))_n \cong H^{d-2}_{G_+} (G_I(M)/x^* G_I(M))_n \rightarrow H^{d-1}_{G_+} (G_I(M))_{n-1}.$$
 From this one can see that $a_{d-1}(G_I(M)) \le 1-d$ implies $a_{d-2}(G_I(N)) \le 2-d$. Since $e_i(I,M) = e_i(I,N)$ for all $i\le 2$, if $(I,M)$ satisfies one of the conditions (iii), (iv) and (v), then the same condition holds for $(I,N)$. Therefore,  $(i)$ holds for $(I,N)$, whence also holds for $(I,M)$.
 
 Finally, if one of conditions (i)...(v) is satisfied, then from the condition (iv) we see that $(I,M)$ satisfies the condition in Lemma \ref{H3a}. Hence $b=1$.
\end{proof}

\medskip
\begin{exm}
Using Example \ref{T3}, we can see that the pair $(I,M)$ satisfies the conditions of Theorem \ref{B2}, where  
$$I=(t^a, t^{a+1}, u_1,...,u_{d-1}) \subset A= k[[t^a, t^{a+1}, t^{a^2-a-1}, u_1,...,u_{d-1}]],$$ ($a\ge 3$, $d\ge 2$) and $M=A$. 
\end{exm}

The above theorem says that if $e_0(I,M) \ge b+2$ and $b\ge 2$, then the inequality in Theorem \ref{B1} is strict. For the case $M=A$, using the bound of Theorem \ref{EB3}, we can give a better bound in the case $b\ge 2$. We need some more preparation.

Recall that the ideal $J\subseteq I$ is called an {\it $M$-reduction} of $I$ if $I^{n+1} M = JI^nM$ for all $n\gg 0$. The  number:
$$r_J(I,M) = \min\{n\ge 0|\  I^{n+1} M = JI^nM\} $$
is called the {\it $M$-reduction number of $I$ with respect to $J$}.   An $M$-reduction of $I$ is called {\it minimal} if it does not strictly contain another $M$-reduction of $I$. The number
$$r(I,M) := \min\{ r_J(I,M)|\ J\ \text{is a minimal  \ } M\text{-reduction of \ } I\}$$
is called the {\it $M$-reduction number } of $I$. The above definitions of reductions and reduction numbers remain valid for any ideal $I$ of a Noetherian ring $R$ and any finitely generated $R$-module $M$.

\begin{rem} \label{AB4}We recall here some facts on reductions. 
 \begin{itemize}
\item[(i)] A minimal $M$-reduction of $I$ is generated by exactly $d$ elements.
\item[(ii)] A minimal $M$-reduction of $I$ can be generated by a maximal $M$-superficial sequence for $I$.
\end{itemize} 
\end{rem}

Below are some  relationships between the reduction number and Hilbert coefficients.

\begin{lem} \label{AB5} Let $M$ be an one-dimensional Cohen-Macaulay module  and $I$ an $\mfr$-primary ideal such that $IM \subseteq \mfr^bM$ for some positive integer $b$. Then
$$r(I,M) \le e_0(I,M) - b.$$
\end{lem}
\begin{proof}
Assume that $x\in I$ is an $M$-superficial element for $I$ such that $r(I,M) = r_{(x)}(I,M)$. 
Then $r(I,M) = r_{(x^*)}(G_+, G_I(M))$.
 By \cite[Proposition 3.2]{Tr}, 
 $$r_{(x^*)}(G_+, G_I(M)) \le \reg(G_I(M)).$$ 
 By Lemma \ref{I6}(i) and (iii), $\reg(G_I(M)) = pn(I,M) \le e_0(I,M) -b$. 
 Hence $r(I,M) \le e_0(I,M) - b$.
 \end{proof}
 
 \begin{lem} \label{AB6} 
 Let $M$ be an one-dimensional Cohen-Macaulay module  and $I$ an $\mfr$-primary ideal. 
 Then
 $$e_2(I,M) \le (r'(I,M) -1)e_1(I,M),$$
 where we set $r'(I,M) := \max\{1, r(I,M)\}$. 
\end{lem}
\begin{proof} 
Assume that $x\in I$ is an $M$-superficial element for $I$ such that $r:= r(I,M) = r_{(x)}(I,M)$. Set $r':= \max\{1, r\}$. 
By \cite[Lemmas 2.1 and 2.2]{RV2}
$$e_1(I,M) = \sum_{j=0}^{r-1}\ell(I^{j+1}M/xI^jM).$$
Hence
$$\begin{array}{ll}
e_2(I,M) &= \sum_{j=1}^{r-1}j \ell(I^{j+1}M/xI^jM) \\ \\
&\le (r'-1)\sum_{j=0}^{r-1}\ell(I^{j+1}M/xI^jM) = (r'-1) e_1(I,M).
\end{array}$$
\end{proof}
Using the above two lemmas, we can give a new bound on $e_2(I,M)$.

\begin{prop} \label{AB7} Let $M$ be a  Cohen-Macaulay module  of dimension $d\ge 2$ and $I$ an $\mfr$-primary ideal such that $IM \subseteq \mfr^bM$ for some positive integer $b$. 
Assume that $e_0(I,M) \ge b+1$. Then
$$e_2(I,M) \le (e_0(I,M) - b -1)e_1(I,M).$$
\end{prop}

\begin{proof} By standard technique, we only need to consider the case $d=2$. 
Let $x\in I$  be an $M$-superficial element for $I$. Set $N= M/xM$. 
Then $N$ is an one-dimensional  Cohen-Macaulay module. 
By the assumption, $e_0(I,M) - b\ge 1$. Hence, by Lemma \ref{AB5},  $r'(I,M) \le e_0(I,M) - b$. 
Applying  Lemma \ref{AB6} to $N$,  by Lemma \ref{I2}(ii) and (iii), we get
\begin{multline*}
e_2(I, N) \le (r'(I,M) -1)e_1(I,M)\\
\le (e_0(I,N)-b-1)e_1(I ,N) = (e_0(I,M) - b-1) e_1(I,M).
\end{multline*}
By (\ref{EB14}), $e_2(I,M) \le e_2(I,N)$. Hence $e_2(I,M) \le (e_0(I,M) - b -1)e_1(I,M).$
\end{proof}

\medskip
Combining the above result with Theorem \ref{EB3} we get the following bound which is clearly better the bound of Theorem \ref{B1} in the case $M=A$ and $b\ge 2$.

\begin{thm} \label{AB8} 
Let $I$ be an $\mfr$-primary ideal of a  Cohen-Macaulay ring $(A,\mfr)$  of dimension $d\ge 2$ and such that $I \subseteq \mfr^b$ for some positive  integer $b$. 
Assume that $e_0(I,M) \ge b+1$. Then
$$e_2(I) \le \frac{3}{2b-1}\binom{e_0(I) - b+1}{3} - (e_0(I)-b-1)\binom{\mu(\mfr)-d}{2}.$$
\end{thm}
\begin{proof} We may assume that $d=2$. For simplicity we set $e_i := e_i(I),\ i=0,1,2$. By Theorem \ref{EB3},
$$e_1 \le \frac{1}{2b-1}\binom{e_0 - b+1}{2} - \binom{\mu(\mfr)-2}{2}.$$
Hence, by Proposition \ref{AB7},
$$\begin{array}{ll}
e_2& \le (e_0- b -1)e_1\\ \\
&\le (e_0- b -1) \left\lbrace \frac{1}{2b-1}\binom{e_0 - b+1}{2} - \binom{\mu(\mfr)-2}{2} \right\rbrace \\ \\
& = \frac{3}{2b-1}\binom{e_0 - b+1}{3} - (e_0-b-1)\binom{\mu(\mfr)-2}{2}.
\end{array}$$
\end{proof}

\noindent {\bf Acknowledgement}.  
 This work is the result of discussions during various stays of all three authors  at Vietnam Institute for Advanced Study in Mathematics (VIASM). The authors would like to thank VIASM for the financial support and generous hospitality.
 The second author is partially supported by PID2019-104844GB-I00, while the third author is partially supported by NCXS02.01/22-23 of Vietnam Academy of Sciences and Technology.

\end{document}